\documentclass[11pt]{article}

\usepackage[a4paper,margin=1in]{geometry}
\usepackage{amsmath,amssymb,amsthm}

\newtheorem{theorem}{Theorem}[section]
\newtheorem{maintheorem}{Theorem}

\newtheorem{proposition}[theorem]{Proposition}
\newtheorem{lemma}[theorem]{Lemma}
\newtheorem{example}[theorem]{Example}
\newtheorem{definition}[theorem]{Definition}
\newtheorem{problem}[theorem]{Problem}
\newtheorem{corollary}[theorem]{Corollary}
\theoremstyle{definition}
\newtheorem{remark}[theorem]{Remark}

\newcommand{\PP}{\mathbb P}
\newcommand{\CC}{\mathbb C}
\newcommand{\OO}{\mathcal O}
\newcommand{\F}{\mathcal F}
\newcommand{\Sing}{\operatorname{Sing}}
\newcommand{\ord}{\operatorname{ord}}
\newcommand{\wt}{\operatorname{wt}}
\newcommand{\Res}{\operatorname{Res}}
\newcommand{\dd}{\,\mathrm d}

\title{Cuspidal Cubic Projective Foliations: Global Rigidity, Deformations, and Pull-Backs}
\author{Bruno Sc\'ardua\\[0.6ex]
\small Institute of Mathematics, Federal University of Rio de Janeiro\\
\small 21945-970 Rio de Janeiro, Brazil\\
\small \texttt{bruno.scardua@gmail.com}}
\date{}

\begin{document}

\maketitle

\begin{abstract}
We study rigidity phenomena for holomorphic foliations associated with
irreducible cuspidal cubics.  If $C\subset\PP^2$ is such a cubic, we
denote by $p$ its cusp and by $q$ its unique smooth inflection point
(or flex), where the tangent line has contact of order three with the
cubic.  Our first main result is an intrinsic global classification through
degree seven: if a foliation $\F$ of degree at most seven leaves $C$
invariant, has singular set $\{p,q\}$, the germ $\F_p$ admits a
nonconstant holomorphic first integral with $C_p$ as its unique
separatrix, and the reduction of $\F_q$ contains no saddle-nodes, then
$\F$ has degree two and is projectively equivalent to the standard
cuspidal foliation
\[
 \dd(x^2+y^3)=0.
\]
In particular it admits a rational first integral.  The proof combines
polynomial normal forms adapted to the invariant cusp, a global contact
constraint at the smooth flex, and weighted local equations arising from
the holomorphic first integral. These ingredients eliminate all cases up
to degree seven; only the extremal degree-seven case requires a final
exact algebraic computation over $\mathbb Q$.

Our second main result concerns degree-two deformations: a sufficiently
small deformation of the standard model which preserves the fixed
cuspidal cubic and has exactly two singular points is trivial.  We also
describe the two basic ways in which cusp-preserving deformations can
escape this rigidity.  Finally, for $n\geq3$ we introduce the cuspidal
linear pull-back locus $\mathcal C_n^{\mathrm{cusp}}$, prove that it is
a proper closed irreducible algebraic subset of the classical linear pull-back
component of $\operatorname{Fol}_2(\PP^n)$, and prove rigidity of the
base foliation for deformations which remain in that pull-back locus.
This leads to an intrinsic pull-back recognition problem and, more
broadly, to the study of geometrically distinguished irreducible
algebraic subloci inside irreducible components of spaces of projective
foliations.
\end{abstract}



\tableofcontents

\section{Introduction}

Let $C\subset\PP^2_{\CC}$ be an irreducible cuspidal cubic.  Besides
its unique cusp, $C$ has a unique smooth inflection point, or
\emph{smooth flex}: this is the unique point $q\in C_{\mathrm{reg}}$
for which the tangent line $T_qC$ has intersection multiplicity
\[
 I_q(C,T_qC)=3.
\]
It is classical that every irreducible cuspidal cubic is projectively
equivalent to
\[
 C_0=\{X^2Z+Y^3=0\};
\]
see, for instance, \cite[Chapter~3]{Dolgachev2012}.  In this standard
model the cusp and the smooth flex are respectively
\[
 p_0=[0:0:1],
 \qquad
 q_0=[1:0:0],
\]
and the tangent line at $q_0$ is $(Z=0)$.  Indeed, the Hessian
determinant of $X^2Z+Y^3$ is a nonzero constant multiple of $X^2Y$, so
$q_0$ is the only smooth point of $C_0$ on its Hessian.  Thus the
marked pair consisting of the cusp and the smooth flex is intrinsic
and is preserved by projective equivalence.

In the affine chart $Z=1$ we write
\[
 f=x^2+y^3.
\]
The standard cuspidal foliation is
\[
 \F_{\mathrm{cusp}}:\quad \dd f=0.
\]
Its projective degree is two and it admits the rational first integral
\[
 \mathcal R([X:Y:Z])
 =
 \frac{X^2Z+Y^3}{Z^3}.
\]

The paper develops three related rigidity questions around this model.
First, we ask to what extent an invariant cuspidal cubic together with
the concentration of the singular set at its cusp and smooth flex
determines a foliation of bounded degree.  Second, once the standard
degree-two model is reached, we ask which cusp-preserving deformations
can keep the singularities concentrated in only two points.  Third, in
higher dimension we study the algebraic locus obtained by linear
pull-back of the planar cuspidal foliation and its rigidity inside the
classical linear pull-back component of the space of foliations.

For later reference, let $C\subset\PP^2$ be an irreducible cuspidal
cubic, with cusp $p$ and unique smooth flex $q$.  We shall call a pair
$(\F,C)$ a \emph{cuspidal pair} if $C$ is invariant by $\F$ and
\[
 \Sing(\F)=\{p,q\}.
\]
Thus the terminology records simultaneously the invariant cubic and the
concentration of the singular set at its two intrinsic distinguished
points.

Our first main result, Theorem~\ref{thm:main}, answers the global question
through degree seven in an intrinsic form; its proof is completed in
Section~\ref{sec:proof-main}.

\begin{maintheorem}[Global rigidity through degree seven]
\label{thm:main}
Let $C\subset\PP^2$ be an irreducible cuspidal cubic, let $p$ be its
cusp, and let $q$ be its unique smooth flex.  Let $\F$ be a holomorphic
foliation of degree $d\leq7$ on $\PP^2$.  Assume that:
\begin{enumerate}\renewcommand{\labelenumi}{\textup{(H\arabic{enumi})}}
\item $C$ is invariant by $\F$;
\item
\[
 \Sing(\F)=\{p,q\};
\]
\item the germ $\F_p$ admits a nonconstant holomorphic first integral
and $C_p$ is its unique separatrix;
\item the reduction of $\F_q$ contains no saddle-nodes.
\end{enumerate}
Under these assumptions, $d=2$ and $\F$ is projectively equivalent to
$\F_{\mathrm{cusp}}$.  More precisely, if $F_C=0$ is a homogeneous
cubic equation of $C$ and $\ell_q=0$ is the tangent line to $C$ at its
smooth flex $q$, then, up to multiplication by a nonzero constant,
\[
 \frac{F_C}{\ell_q^3}
\]
is a rational first integral of $\F$.
\end{maintheorem}

For the proof we use the projective classification just recalled.
Choose $A\in\operatorname{PGL}(3,\CC)$ with $A(C)=C_0$.  Since the cusp
and the unique smooth flex are intrinsic, necessarily
\[
 A(p)=p_0,
 \qquad
 A(q)=q_0.
\]
Replacing $\F$ by $A_*\F$, we may therefore work throughout the global
part of the paper with the normalized configuration
\[
 C=C_0,
 \qquad
 p=p_0,
 \qquad
 q=q_0.
\]
All hypotheses of the theorem are invariant under this projective
change of coordinates.

The proof of Theorem~\ref{thm:main} is organized in three stages. We
first obtain a polynomial normal form adapted to the invariant cusp.
We then translate the global two-singularity condition into an exact
intersection multiplicity at $q$. Finally, the local first integral at
$p$ gives a weighted cohomological equation. The combination of these
constraints excludes all nonstandard foliations of degrees at most
seven.

Our second main result, Theorem~\ref{thm:two-singularity-deformation-final},
established in Section~\ref{sec:planar-deformation}, shows a stronger
rigidity phenomenon in degree two.  For a sufficiently
small deformation of $\F_{\mathrm{cusp}}$, the local first-integral
condition at the cusp and the reduction hypothesis at the flex are no
longer needed: preservation of the fixed cuspidal cubic together with
the requirement that the singular set consist of exactly two points
already forces
\[
 \F_t=\F_{\mathrm{cusp}}.
\]
Two explicit families show the sharp geometric mechanism behind this
statement: a nontrivial cusp-preserving deformation must lose the
concentration of singularities either near the affine cusp or at
infinity.

The third part of the paper passes to $\PP^n$, $n\geq3$.  We define the
cuspidal linear pull-back locus $\mathcal C_n^{\mathrm{cusp}}$ and prove
in Proposition~\ref{prop:cuspidal-locus-irreducible-final} that it is a
proper closed irreducible algebraic subset of
$\operatorname{Fol}_2(\PP^n)$ contained in the classical linear
pull-back irreducible component.  Theorem~\ref{thm:rigidity-inside-pullback-final},
proved in Section~\ref{sec:pullback-locus}, then shows that, for deformations
already known to remain linear
pull-backs, an invariant holomorphic family of cuspidal cubics together
with the two-singularity condition freezes the base foliation up to
projective equivalence.  This higher-dimensional result motivates an
intrinsic pull-back recognition problem: can the invariant cuspidal
cone and the two distinguished transverse singularity types themselves
force the linear pull-back structure?  It also suggests a broader
algebro-geometric program: after the irreducible components of a space
of foliations are known, one may study the natural irreducible
algebraic subloci inside those components which are selected by
geometric or dynamical conditions.

\noindent{\bf Acknowledgement}. During the preparation of this work, the  author was partially funded by 
Fundação Getúlio Vargas - Rio de Janeiro.  
\section{Polynomial normal forms}
\label{sec:normal}

Throughout the next sections, $C\subset\PP^2$ denotes the normalized
cuspidal cubic
\[
 C=\{X^2Z+Y^3=0\},
\]
with cusp
\[
 p=[0:0:1]
\]
and distinguished smooth flex
\[
 q=[1:0:0].
\]
We work in the affine chart $Z=1$, where
\[
 f=x^2+y^3,
 \qquad
 \eta=3x\,\dd y-2y\,\dd x.
\]
We use the normalization
\[
 \gamma(t)=(t^3,-t^2)
\]
of the affine cusp. Along $C$,
\begin{equation}
 \dd f=t\eta.
 \label{eq:dfteta}
\end{equation}

\begin{lemma}[Logarithmic decomposition]
\label{lem:decomp}
Let
\[
 \omega=a\,\dd x+b\,\dd y
\]
be a polynomial one-form leaving $(f=0)$ invariant. Then there exist
unique polynomials $A,B\in\CC[x,y]$ such that
\[
 \omega=A\,\dd f+B\eta.
\]
\end{lemma}

\begin{proof}
Invariance is equivalent to
\[
 f\mid 3y^2a-2xb.
\]
Set
\[
 B=\frac{2xb-3y^2a}{6f}.
\]
Then $B$ is polynomial. Moreover
\[
 3y^2(a+2yB)=2x(b-3xB).
\]
Since $\gcd(x,y^2)=1$, the polynomial $a+2yB$ is divisible by $x$, and
\[
 A=\frac{a+2yB}{2x}
\]
is polynomial. Direct substitution gives the asserted decomposition.

To prove uniqueness, suppose that
\[
 A\,\dd f+B\eta=A'\,\dd f+B'\eta.
\]
Then
\[
 (A-A')\,\dd f+(B-B')\eta=0.
\]
Wedge this identity with $\eta$. Since
\[
 \dd f\wedge\eta
 =
 6f\,\dd x\wedge\dd y
\]
and $\CC[x,y]$ is an integral domain, we obtain
\[
 (A-A')f=0,
\]
hence $A=A'$. The preceding identity then reduces to
\[
 (B-B')\eta=0,
\]
and therefore $B=B'$. Thus the decomposition is unique.
\end{proof}

\begin{lemma}[The local first integral at the cusp]
\label{lem:cuspFI}
Under \textup{(H3)}, a reduced first integral at $p$
has the form
\[
 H=uf,
 \qquad
 f=x^2+y^3,
 \qquad
 u\in\OO_{\CC^2,0}^{*}.
\]
Furthermore
\[
 \mu_p(\F)=2.
\]
\end{lemma}

\begin{proof}
The zero divisor of a reduced first integral is the union of the local
separatrices. Since $C_p$ is the unique separatrix and $f$ is reduced
and irreducible,
\[
 H=uf
\]
for a holomorphic unit $u$.

The foliation is therefore locally defined by a unit multiple of
$\dd(uf)$, and hence
\[
 \mu_p(\F)
 =
 \dim_{\CC}
 \frac{\CC\{x,y\}}{(f_x,f_y)}
 =
 \dim_{\CC}
 \frac{\CC\{x,y\}}{(x,y^2)}
 =
 2.
\]
\end{proof}

We also recall a standard consequence of the Mattei--Moussu theory: a germ of
holomorphic foliation in $(\CC^2,0)$ admitting a nonconstant holomorphic first
integral has no saddle-nodes in its reduction; see
\cite{MatteiMoussu1980}.  For the analytic first-integral theorem itself, see
also \cite{MatteiMoussu}.

\begin{proposition}[Affine polynomial normal form]
\label{prop:KL}
Let
\[
 \omega=a\,\dd x+b\,\dd y
\]
be a polynomial $1$-form defining, in the affine chart $Z=1$, a foliation
$\F$ satisfying the standing hypotheses of Theorem~\ref{thm:main}. Then,
after multiplying $\omega$ by a nonzero constant, there exist polynomials
$K,L$ such that
\begin{equation}
 \omega
 =
 \dd f+f(2K\,\dd x+3L\,\dd y).
 \label{eq:KL}
\end{equation}
Equivalently,
\[
 \omega=2P\,\dd x+3Q\,\dd y,
\]
where
\[
 P=x+fK,
 \qquad
 Q=y^2+fL.
\]
Moreover
\begin{equation}
 (P,Q)=(x,y^2)
 \qquad\text{in }\CC[x,y].
 \label{eq:singideal}
\end{equation}
\end{proposition}

\begin{proof}
Write
\[
 \omega=A\,\dd f+B\eta
\]
as in Lemma~\ref{lem:decomp}. On the normalization of $C$,
\[
 \omega|_C
 =
 \bigl(tA(\gamma(t))+B(\gamma(t))\bigr)\eta.
\]
By Lemma~\ref{lem:cuspFI}, this coefficient has a simple zero at $t=0$.
Hypothesis \textup{(H2)} implies that it has no other finite zero.
Thus, after multiplication by a constant,
\[
 tA(t^3,-t^2)+B(t^3,-t^2)=t.
\]
Consequently
\[
 x(A-1)-yB
 \quad\text{and}\quad
 y^2(A-1)+xB
\]
vanish on $(f=0)$ and are divisible by $f$. Put
\[
 K=\frac{x(A-1)-yB}{f},
 \qquad
 L=\frac{y^2(A-1)+xB}{f}.
\]
Then
\[
 A=1+xK+yL,
 \qquad
 B=xL-y^2K.
\]
Using
\[
 x\,\dd f-y^2\eta=2f\,\dd x,
 \qquad
 y\,\dd f+x\eta=3f\,\dd y,
\]
we obtain \eqref{eq:KL}.

It remains to prove \eqref{eq:singideal}. The algebra
\[
 R=\CC[x,y]/(P,Q)
\]
is supported only at the origin by \textup{(H2)} and has length
\[
 \dim_{\CC}R=\mu_p=2.
\]
Hence its maximal ideal satisfies $\mathfrak m_R^2=0$. In $R$,
\[
 \bar x=-\bar f\,\bar K\in\mathfrak m_R^2,
 \qquad
 \bar y^2=-\bar f\,\bar L\in\mathfrak m_R^2.
\]
Therefore
\[
 \bar x=\bar y^2=0.
\]
Thus $(x,y^2)\subset(P,Q)$, while the opposite inclusion follows
immediately from the definitions of $P,Q$.
\end{proof}

We next use the absence of saddle-nodes at infinity. We briefly recall
the connection that enters the argument. Let $S$ be an invariant
irreducible curve of a holomorphic foliation $\mathcal F$. On
\[
 S^{\circ}=S_{\mathrm{reg}}\setminus\Sing(\mathcal F),
\]
the curve is a leaf of $\mathcal F$, and its normal bundle carries the
canonical Bott partial connection; for the standard construction and the
formula below, see Tondeur \cite[Chapter~5, formula~(5.1)]{Tondeur1988}.
If $X$ is a local vector field tangent to $\mathcal F$ along $S^{\circ}$
and $\overline Y$ is a local section of the normal bundle, represented by
a vector field $Y$, then
\[
 \nabla_X^{\mathrm B}(\overline Y)=\overline{[X,Y]}.
\]
After choosing a local trivialization of the normal bundle, this partial
connection is represented along $S^{\circ}$ by a holomorphic $1$-form;
in the singular holomorphic setting considered here it extends
meromorphically to the normalization of $S$.  This meromorphic behavior
near a singular invariant branch is the one underlying the
Camacho--Sad index; see \cite{CamachoSad}. If
\[
 \nu:\widetilde S\longrightarrow S
\]
is the normalization, we shall use the pull-back connection
$\nu^*\nabla^{\mathrm B}$ on $\nu^*N_S$.

\begin{lemma}[Logarithmic Bott pole]
\label{lem:bott}
Let $S$ be an invariant irreducible branch of a germ of foliation whose
reduction contains no saddle-nodes, and let
\[
 \nu:\widetilde S\longrightarrow S
\]
be its normalization. Then the pull-back Bott connection
$\nu^*\nabla^{\mathrm B}$ has at most a logarithmic pole at the point of
$\widetilde S$ lying over the singularity.
\end{lemma}

\begin{proof}
By a simultaneous resolution of $(\mathcal F,S)$ we mean a finite
composition of point blow-ups
\[
 \pi:(M,D)\longrightarrow(\CC^2,0)
\]
for which the strict transform $\widetilde S$ is smooth, the union of
$\widetilde S$ with the exceptional divisor has normal crossings, and
the transformed foliation has only reduced singularities along this
divisor.  Such a modification exists by combining the embedded
resolution of the plane branch with Seidenberg's reduction theorem and
passing to a common refinement; see \cite{Seidenberg}.  Since $S$ is
invariant, its strict transform remains invariant by the transformed
foliation.

At a regular point the Bott connection is holomorphic. At a reduced
nondegenerate singularity choose coordinates $(z,s)$ such that the
strict transform of the branch is $(z=0)$ and the foliation is defined by
\[
 z\,a(z,s)\,\dd s-s\,b(z,s)\,\dd z,
 \qquad
 a(0,0)b(0,0)\neq0.
\]
Relative to the normal section $z$, the connection is
\[
 -\frac{a(0,s)}{b(0,s)}\,\frac{\dd s}{s},
\]
hence has a simple pole.

Under blow-up, and under multiplication of a local defining section by
a unit, the connection changes by logarithmic differentials. Thus a
pole of order greater than one cannot arise along the strict transform
unless a saddle-node occurs. By hypothesis, no saddle-node occurs in
the reduction.
\end{proof}

\begin{proposition}[Polynomial Bott normal form]
\label{prop:UV}
Let $\omega$ be the polynomial $1$-form in the affine chart $Z=1$ defining
the foliation $\F$ under the standing hypotheses of
Theorem~\ref{thm:main}. Then there exist polynomials $U,V\in\CC[x,y]$ such
that
\begin{equation}
 \boxed{
 \omega
 =
 (1+fU)\,\dd f+fV\eta.
 }
 \label{eq:UV}
\end{equation}
They satisfy
\begin{equation}
 K=xU-yV,
 \qquad
 L=y^2U+xV.
 \label{eq:UVKL}
\end{equation}
\end{proposition}

\begin{proof}
Put
\[
 \alpha=2K\,\dd x+3L\,\dd y.
\]
The restriction of $\alpha$ to the normalization of $C$ is the Bott
connection in the affine trivialization supplied by $f$, and
\[
 \gamma^*\alpha
 =
 6t\bigl(tK(t^3,-t^2)-L(t^3,-t^2)\bigr)\,\dd t.
\]
This is a polynomial one-form in $t$.

Near $q$, in coordinates
\[
 u=\frac YX,
 \qquad
 v=\frac ZX,
 \qquad
 w=v+u^3,
\]
one has
\[
 f=\frac{w}{v^3}.
\]
Thus replacing the affine section $f$ by the holomorphic defining
section $w$ changes the connection only by a logarithmic differential.
Along $C$,
\[
 v=-u^3,
\]
so this correction has at most a logarithmic pole at infinity.

By Lemma~\ref{lem:bott}, $\gamma^*\alpha$ has at most a simple pole at
$t=\infty$. A nonzero polynomial one-form $h(t)\,\dd t$ has a pole of
order at least two at infinity. Hence
\[
 \gamma^*\alpha=0.
\]
Therefore
\[
 xL-y^2K=0
 \quad\text{on }C.
\]
It follows that $xL-y^2K$ is divisible by $f$. The relations obtained
in the proof of Proposition~\ref{prop:KL} then also show that $A-1$ is divisible
by $f$. Write
\[
 A=1+fU,
 \qquad
 B=fV.
\]
This gives \eqref{eq:UV}, and solving for $K,L$ gives
\eqref{eq:UVKL}.
\end{proof}

\section{The global contact at infinity}
\label{sec:infinity}

Let $r\in\Sing(\F)$ be an isolated singularity.  If, in local
coordinates centered at $r$, the foliation is represented by a reduced
holomorphic $1$-form
\[
 \omega=A(x,y)\,\dd x+B(x,y)\,\dd y,
 \qquad \gcd(A,B)=1,
\]
its \emph{Milnor number} is
\[
 \mu_r(\F)
 :=
 \dim_{\CC}\frac{\OO_{\PP^2,r}}{(A,B)}.
\]
The dimension is finite because the singularity is isolated, and the
number is independent of the chosen local coordinates and of
multiplication of $\omega$ by a holomorphic unit; see, for instance,
\cite{Brunella2015}.

We shall use the standard global Milnor-number, or Chern-class, formula
for a degree-$d$ foliation on $\PP^2$ with isolated singularities:
\begin{equation}
 \sum_{r\in\Sing(\F)}\mu_r(\F)
 =
 d^2+d+1.
 \label{eq:global-milnor-formula}
\end{equation}
It follows by computing the second Chern class of the singular scheme of
the foliation; see, for instance, \cite{Brunella2015}.
Since $\mu_p=2$ and $\Sing(\F)=\{p,q\}$,
\begin{equation}
 \mu_q=d^2+d-1.
 \label{eq:muqglobal}
\end{equation}

Put
\[
 m=\max\{\deg K,\deg L\}.
\]
If $K=L=0$, then $\omega=\dd f$ and we are done. Assume henceforth
that $K,L$ do not both vanish.

\begin{lemma}[Leading homogeneous polynomial]
\label{lem:J}
Let $U,V\in\CC[x,y]$ be the polynomials appearing in the polynomial Bott
normal form \eqref{eq:UV} of Proposition~\ref{prop:UV}, and let
\[
 m=\max\{\deg K,\deg L\},
\]
where $K,L$ are the polynomials introduced in
Proposition~\ref{prop:KL}. Then
\[
 \deg U\leq m-2,
 \qquad
 \deg V\leq m-1.
\]
Moreover, there exists a homogeneous polynomial $J_{m-3}$ of degree
$m-3$ such that
\begin{equation}
 U_{m-2}=xJ_{m-3},
 \qquad
 V_{m-1}=-3y^2J_{m-3}.
 \label{eq:J}
\end{equation}
In particular
\[
 m\geq3
\]
for every nonstandard foliation, and its projective degree is
\begin{equation}
 d=m+2.
 \label{eq:dm2}
\end{equation}
\end{lemma}

\begin{proof}
The top homogeneous part of the affine form satisfies the projective
radial relation
\[
 2xK_m+3yL_m=0.
\]
Hence
\[
 K_m=3yH_{m-1},
 \qquad
 L_m=-2xH_{m-1}.
\]
The identities \eqref{eq:UVKL} imply
\[
 \deg U\leq m-2,
 \qquad
 \deg V\leq m-1.
\]
Taking the homogeneous part of degree $m$ gives
\[
 -yV_{m-1}=3yH_{m-1},
\]
\[
 y^2U_{m-2}+xV_{m-1}=-2xH_{m-1}.
\]
Thus
\[
 V_{m-1}=-3H_{m-1},
 \qquad
 y^2U_{m-2}=xH_{m-1}.
\]
Since $\gcd(x,y^2)=1$,
\[
 H_{m-1}=y^2J_{m-3},
\]
which proves \eqref{eq:J}.

If $m<3$, the highest homogeneous part must vanish, contradicting the
definition of $m$. Finally, the affine form has ordinary degree $m+3$
and its top homogeneous part satisfies the radial relation, so its
projective degree is $m+2$.
\end{proof}

Near
\[
 q=[1:0:0]
\]
use
\[
 u=Y/X,
 \qquad
 v=Z/X,
 \qquad
 w=v+u^3.
\]
Define
\[
 \widehat U=v^{m-2}U(1/v,u/v),
 \qquad
 \widehat V=v^{m-1}V(1/v,u/v),
\]
and
\[
 P_\infty=\frac{3u^2\widehat U+\widehat V}{v}.
\]
The numerator is divisible by $v$ by \eqref{eq:J}.

Put
\begin{equation}
 T=v^{m+1}+w\widehat U,
 \qquad
 G=3u^2v^m+wP_\infty.
 \label{eq:TG}
\end{equation}

For two germs $A,B\in\OO_{\CC^2,0}$ having no common irreducible
component, their \emph{local intersection multiplicity at the origin}
is
\[
 I_0(A,B)
 :=
 \dim_{\CC}\frac{\OO_{\CC^2,0}}{(A,B)}.
\]
This dimension is finite precisely when the curve germs $(A=0)$ and
$(B=0)$ have no common irreducible component.  If $(A=0)$ is an
irreducible branch with a primitive parametrization
$\gamma(t)=(u(t),w(t))$ and $B$ does not vanish identically on it, then
\[
 I_0(A,B)=\ord_{t=0}B(\gamma(t)).
\]
We shall also use the standard additivity properties of local
intersection multiplicity.  See, for example,
\cite[Chapters~2 and~3]{Wall2004}.

We now express the polynomial Bott normal form \eqref{eq:UV} in the
local coordinates $(u,w)$ at $q$.  After clearing denominators, a
common factor $v=w-u^3$ appears.  Removing divisorial common factors
from a local defining $1$-form is the usual saturation operation; the
resulting coefficients have no common nonunit factor in
$\OO_{\CC^2,0}$.

\begin{proposition}[Local form and contact formula]
\label{prop:contact}
The germ of $\F$ at $q=[1:0:0]$ is defined by the saturated
holomorphic $1$-form
\begin{equation}
 \Theta
 =
 3wG\,\dd u-(2T+uG)\,\dd w.
 \label{eq:Theta}
\end{equation}
Moreover,
\begin{equation}
 \mu_q=3m+3+I_0(G,T).
 \label{eq:mucontact}
\end{equation}
Consequently,
\begin{equation}
 \boxed{
 I_0(G,T)=m^2+2m+2.
 }
 \label{eq:requiredcontact}
\end{equation}
\end{proposition}

\begin{proof}
Since
\[
 f=\frac{w}{v^3},
\]
one computes
\[
 \dd f
 =
 v^{-4}\bigl((v-3w)\,\dd w+9wu^2\,\dd u\bigr),
\]
and
\[
 \eta=v^{-3}(3w\,\dd u-u\,\dd w).
\]
Substituting these expressions into \eqref{eq:UV} and multiplying by
$v^{m+5}$ gives
\[
 v\bigl(3wG\,\dd u-(2T+uG)\,\dd w\bigr).
\]
Recall that if a holomorphic $1$-form
\[
 \omega_0=A\,\dd u+B\,\dd w
\]
has a common factor $h$, so that $A=hA_1$ and $B=hB_1$, then away from
$(h=0)$ the forms $\omega_0$ and
\[
 \omega_1=A_1\,\dd u+B_1\,\dd w
\]
define the same tangent distribution. The holomorphic foliation obtained by
extending this distribution across $(h=0)$ is represented by $\omega_1$.
Repeating this division until the coefficients are relatively prime gives
the saturated local defining form, unique up to multiplication by a
holomorphic unit. Thus the common factor $v$ above is removed, giving
\eqref{eq:Theta}.

There is no further common divisor.  Indeed, let $h$ be an
irreducible common divisor of $wG$ and $2T+uG$.  If $h=w$, then $w$
would divide $2T+uG$, whereas
\[
 (2T+uG)(u,0)=(-1)^m u^{3m+3}\neq0.
\]
Hence $h\neq w$, so $h$ divides $G$; since it also divides $2T+uG$, it
then divides $T$.  This would produce a positive-dimensional singular
locus, contrary to \textup{(H2)}.

On $w=0$ one has $v=-u^3$, and
\[
 (2T+uG)(u,0)=(-1)^m u^{3m+3}.
\]
Hence
\[
 \mu_q
 =
 I_0(wG,2T+uG)
 =
 I_0(w,2T+uG)+I_0(G,2T+uG).
\]
The first term is $3m+3$, while
\[
 (G,2T+uG)=(G,T).
\]
Thus
\[
 \mu_q=3m+3+I_0(G,T).
\]
Using $d=m+2$ in \eqref{eq:muqglobal} gives
\eqref{eq:requiredcontact}.
\end{proof}

We shall also use the identity
\begin{equation}
 vG=3u^2T+w\widehat V.
 \label{eq:fundamentalidentity}
\end{equation}

\section{The weighted cusp equation}
\label{sec:weighted}

Let
\[
 X_0=3y^2\partial_x-2x\partial_y,
 \qquad
 E_0=3x\partial_x+2y\partial_y.
\]
A tangent vector field for \eqref{eq:UV} is
\[
 Y=(1+fU)X_0+fV E_0.
\]
Since
\[
 X_0(f)=0,
 \qquad
 E_0(f)=6f,
\]
the local first integral
\[
 H=fe^\varphi
\]
satisfies
\begin{equation}
 Y(\varphi)=-6fV.
 \label{eq:cohom}
\end{equation}

Use the cusp weights
\[
 \wt(x)=3,
 \qquad
 \wt(y)=2.
\]
For $N\ge0$, let
\[
 \mathcal H_N
 :=
 \left\{
 P\in\CC[x,y];\;
 P(\lambda^3x,\lambda^2y)=\lambda^N P(x,y)
 \text{ for all }\lambda\in\CC^*
 \right\}.
\]
Equivalently,
\[
 \mathcal H_N
 =
 \operatorname{span}_{\CC}
 \{x^iy^j;\ 3i+2j=N\}.
\]
For a polynomial or convergent power series $P$, we write $[P]_N$ for
its weighted homogeneous component belonging to $\mathcal H_N$.

\begin{lemma}[Weighted recurrence]
\label{lem:recurrence}
Write
\[
 \varphi=\sum_{r\geq1}\varphi_r,
 \qquad
 \varphi_r\in\mathcal H_r.
\]
The component of target weight $N$ in \eqref{eq:cohom} is
\begin{equation}
 X_0(\varphi_{N-1})
 =
 [-6fV]_N
 -
 \sum_{a+r+1=N}[fU]_aX_0(\varphi_r)
 -
 \sum_{b+r=N}[fV]_bE_0(\varphi_r).
 \label{eq:recurrence}
\end{equation}
Furthermore
\[
 X_0:\mathcal H_{N-1}\longrightarrow\mathcal H_N
\]
is surjective unless
\[
 N\equiv0,2\pmod6,
\]
in which case its cokernel is one-dimensional.
\end{lemma}

\begin{proof}
Write the weighted homogeneous decompositions
\[
 \varphi=\sum_{r\geq1}\varphi_r,
 \qquad
 fU=\sum_{a\geq0}[fU]_a,
 \qquad
 fV=\sum_{b\geq0}[fV]_b.
\]
For the cusp weights
\[
 \wt(x)=3,
 \qquad
 \wt(y)=2,
\]
the vector field
\[
 X_0=3y^2\partial_x-2x\partial_y
\]
has weighted degree $1$, whereas
\[
 E_0=3x\partial_x+2y\partial_y
\]
has weighted degree $0$. Hence
\[
 X_0(\varphi_r)\in\mathcal H_{r+1},
 \qquad
 E_0(\varphi_r)\in\mathcal H_r.
\]
Expanding
\[
 (1+fU)X_0(\varphi)+fV E_0(\varphi)=-6fV
\]
and taking the weighted homogeneous component of weight $N$ gives
\[
 X_0(\varphi_{N-1})
 +\sum_{a+r+1=N}[fU]_aX_0(\varphi_r)
 +\sum_{b+r=N}[fV]_bE_0(\varphi_r)
 =[-6fV]_N.
\]
Solving for the first term gives precisely \eqref{eq:recurrence}.

Since $X_0(f)=0$, the kernel of $X_0$ on weighted homogeneous
polynomials is generated by powers of $f$. Counting monomials of a
fixed cusp weight and applying rank-nullity gives the cokernel
statement.

The ambiguity in $\varphi$ coming from a kernel term $cf^j$ can be
removed inductively by replacing
\[
 H
 \longmapsto
 H\exp(-cH^j).
\]
This changes $\varphi$ first in weight $6j$ and leaves all previously
determined terms unchanged.
\end{proof}

\section{Degrees five and six}
\label{sec:d56}

We now begin the degree-by-degree exclusion. Recall that the standard
cuspidal foliation corresponds to
\[
 K=L=0,
\]
in which case $\omega=\dd f$ and the projective degree is $2$. Suppose
therefore that the foliation is nonstandard, so that $K$ and $L$ do not both
vanish, and put
\[
 m=\max\{\deg K,\deg L\}.
\]
By Lemma~\ref{lem:J}, the leading homogeneous structure forces
\[
 m\geq3,
\]
and the projective degree is
\[
 d=m+2.
\]
Consequently every nonstandard foliation in our class has $d\geq5$. In
particular, there are no nonstandard foliations of exact degree three or
four. Thus, in order to prove Theorem~\ref{thm:main} for $d\leq7$, it
remains only to exclude
\[
 d=5,\qquad d=6,\qquad d=7,
\]
corresponding respectively to
\[
 m=3,\qquad m=4,\qquad m=5.
\]

\begin{proposition}[Degree five]
\label{prop:d5}
There is no exact degree-five foliation satisfying
\textup{(H1)--(H4)}.
\end{proposition}

\begin{proof}
Here $m=3$, so $J_0=c\neq0$. The degree bounds and \eqref{eq:J} give
\[
 U=cx+\kappa
\]
and, initially,
\[
 V=-3cy^2+\ell x+\alpha y+\beta.
\]
The target-weight-$6$ and target-weight-$8$ equations in
\eqref{eq:recurrence} give
\[
 \beta=0,
 \qquad
 \alpha=0.
\]
Thus
\[
 U=cx+\kappa,
 \qquad
 V=-3cy^2+\ell x.
\]

Using \eqref{eq:UVKL},
\[
 K=cx^2-\ell xy+\kappa x+3cy^3,
\]
\[
 L=\ell x^2+\kappa y^2-2cxy^2.
\]
A direct exact computation gives
\[
 \Res_x(P,Q)
 =
 y^2
 \left(
 27c^4y^6
 +3c^2\ell^2y^5
 +9c^2\ell\kappa y^4
 +\ell^3\kappa y^3
 +\ell^3
 \right).
\]
If $\ell\neq0$, the factor in parentheses has a nonzero complex root,
giving an affine singularity distinct from $p$.

If $\ell=0$, then
\[
 Q(x,0)=0,
\]
whereas
\[
 P(x,0)
 =
 x(1+\kappa x^2+cx^3).
\]
Since $c\neq0$, the polynomial in parentheses has a nonzero complex
root, again contradicting \textup{(H2)}.
\end{proof}

\begin{proposition}[Degree six]
\label{prop:d6}
There is no exact degree-six foliation satisfying
\textup{(H1)--(H4)}.
\end{proposition}

\begin{proof}
Here $m=4$ and
\[
 J_1=ax+by.
\]

If $a\neq0$, then
\[
 T_w(0,0)\neq0.
\]
Thus $T=0$ is a smooth branch at the origin, and the direct substitution
in \eqref{eq:fundamentalidentity} gives
\[
 I_0(G,T)=14.
\]
On the other hand, \eqref{eq:requiredcontact} requires
\[
 I_0(G,T)=4^2+2\cdot4+2=26,
\]
a contradiction. Therefore
\[
 a=0,
 \qquad
 b\neq0.
\]

The generic forms allowed by the degree bounds are
\[
 U
 =
 bxy+u_{10}x+u_{01}y+u_{00},
\]
and
\[
\begin{split}
 V={}&-3by^3
 +v_{20}x^2+v_{11}xy+v_{02}y^2\\
 &+v_{10}x+v_{01}y+v_{00}.
\end{split}
\]
The target weights $6,8,12$ in \eqref{eq:recurrence} give
\[
 v_{00}=0,
 \qquad
 v_{01}=0,
 \qquad
 v_{20}=2b.
\]

At $q$ one obtains
\[
 G_w(0,0)=2b\neq0,
\]
and
\[
 G(u,0)=3u^{14}.
\]
Therefore $G=0$ has a unique smooth branch
\[
 w=-\frac{3}{2b}u^{14}+O(u^{15}).
\]
Substitution in $T$ gives
\[
 T
 =
 -\frac52u^{15}+O(u^{16}).
\]
Hence
\[
 I_0(G,T)=15.
\]
The required value is $26$, a contradiction.
\end{proof}

\section{The nonextremal degree-seven sectors}
\label{sec:d7nonext}

Let now $m=5$, so $d=7$. Write
\[
 J_2=ax^2+bxy+cy^2.
\]
Put
\[
 k=\ord_{u=0}J_2(1,u).
\]
Thus
\[
 k=
 \begin{cases}
 0,&a\neq0,\\
 1,&a=0,\ b\neq0,\\
 2,&a=b=0,\ c\neq0.
 \end{cases}
\]
We shall refer to the cases determined by the value of $k$ as the
\emph{$k$-sectors} of the degree-seven analysis. Thus a sector is simply the
stratum of possible leading homogeneous polynomials $J_2$ characterized by
the order of vanishing of $J_2(1,u)$ at $u=0$. The cases $k=0,1$ are called
\emph{nonextremal sectors}, whereas $k=2$ is called the \emph{extremal
sector}. The latter terminology reflects the fact that $k=2$ is the maximal
possible vanishing order of the nonzero quadratic polynomial $J_2(1,u)$ at
$u=0$; in that case
\[
 J_2=cy^2,
 \qquad c\neq0.
\]

The global requirement is
\begin{equation}
 I_0(G,T)=37.
 \label{eq:d7contact}
\end{equation}

\begin{lemma}[The sector $k=0$]
\label{lem:d7k0}
Assume $d=7$, so that $m=5$, and write
\[
 J_2=ax^2+bxy+cy^2.
\]
In the sector $k=0$, equivalently $a=J_2(1,0)\neq0$, one has
\[
 I_0(G,T)=17.
\]
On the other hand, the global contact formula
\eqref{eq:requiredcontact}, with $m=5$, requires
\[
 I_0(G,T)=5^2+2\cdot5+2=37.
\]
Hence the sector $k=0$ is impossible.
\end{lemma}

\begin{proof}
Since $J_2(1,0)\neq0$,
\[
 T_w(0,0)\neq0.
\]
Thus $T=0$ is smooth. Its solution has
\[
 \ord_u w=18,
 \qquad
 \ord_u v=3.
\]
Moreover
\[
 \ord_u\widehat V(u,0)=2.
\]
Using \eqref{eq:fundamentalidentity} on $T=0$,
\[
 G=\frac{w\widehat V}{v},
\]
and therefore
\[
 I_0(G,T)=18+2-3=17.
\]
\end{proof}

Recall from \eqref{eq:TG} that, in the local coordinates
\[
 u=Y/X,
 \qquad
 v=Z/X,
 \qquad
 w=v+u^3
\]
near $q$, the function $T$ is
\[
 T(u,w)=v^{m+1}+w\widehat U,
 \qquad
 v=w-u^3.
\]
In the present degree-seven case $m=5$, so
\[
 T=v^6+w\widehat U.
\]
We regard $T$ as a polynomial in the local variables $(u,w)$. If
\[
 T(u,w)=\sum_{i,j\geq0}c_{ij}u^iw^j,
\]
its Newton polygon is the lower convex boundary of
\[
 \operatorname{Conv}\left(
 \bigcup_{c_{ij}\neq0}
 \bigl((i,j)+\mathbb R_{\geq0}^2\bigr)
 \right).
\]
Its compact edges will be called the Newton faces of $T$. We use only this
elementary form of the Newton-polygon construction below.

\begin{lemma}[The sector $k=1$]
\label{lem:d7k1}
Assume $k=1$, equivalently $a=0$ and $b\neq0$, so that
\[
 J_2(1,u)=\lambda u+O(u^2),
 \qquad
 \lambda\neq0.
\]
Then
\[
 I_0(G,T)\leq18.
\]
Since the global contact formula requires $I_0(G,T)=37$, the sector $k=1$
is impossible.
\end{lemma}

\begin{proof}
In the sector $k=1$, one compact face of the Newton polygon of $T$
corresponds to the balance between the term $v^6$ and the leading term
determined by $J_2(1,u)$. It is detected by the toric scaling
\[
 u=\rho,
 \qquad
 w=\rho^{17}z,
\]
and we call it the \emph{cusp-adjacent face}, since its associated branch
has the high contact dictated by the cusp at infinity. The remaining
compact face on the lower Newton boundary will be called the
\emph{complementary face}.

We first consider the cusp-adjacent branch. Put
\[
 r_0=3m+3-k=17
\]
and use the birational toric chart
\[
 u=\rho,
 \qquad
 w=\rho^{17}z.
\]
Since
\[
 v=\rho^3(-1+O(\rho)),
\]
one has
\[
 T
 =
 \rho^{18}\bigl(F(z)+O(\rho)\bigr),
\]
\[
 G
 =
 \rho^{17}\bigl(P(z)+O(\rho)\bigr),
\]
with
\[
 F(z)=1+\lambda z,
 \qquad
 P(z)=-3+\alpha z
\]
for some $\alpha\in\CC$.

If $F$ and $P$ had a common root $z_0\in\CC^*$, then
\[
 P=-3F.
\]
Pulling back \eqref{eq:Theta} and saturating gives at the face level
\[
 8zF(z)\,\dd\rho+\rho F(z)\,\dd z.
\]
Since $F$ has a simple nonzero root, the transformed singularity at
$(0,z_0)$ has one zero and one nonzero eigenvalue. It is therefore a
saddle-node, contradicting \textup{(H4)}.

Consequently there is no cancellation on this branch, and its
contribution to $I_0(G,T)$ is at most
\[
 17.
\]

We now turn to the complementary branch. The support of $\widehat U$
contains the term $\lambda u$. All remaining terms give, in the
complementary region, pure $w$-orders $1,2,3$, or the model order $5$.
Let
\[
 j\in\{1,2,3,5\}
\]
be determined by the first pure $w$-term occurring on the lower
boundary. The corresponding primitive toric chart is
\[
 u=\rho^jz,
 \qquad
 w=\rho.
\]
The face of $T$ is
\[
 T
 =
 \rho^{j+1}
 \bigl(A+\lambda z+O(\rho)\bigr),
 \qquad
 A\neq0.
\]
Thus the complementary branch has a unique simple toric root
\[
 z_0=-A/\lambda.
\]

Let $\tau$ be the coefficient of $v$ in $\widehat V$, equivalently the
constant term of $P_\infty$. If $\tau\neq0$, then on the complementary
branch
\[
 G=\tau\rho+O(\rho^2),
\]
so this branch contributes exactly one to $I_0(G,T)$.

Suppose instead that $\tau=0$. Then
\[
 G=O(\rho^2).
\]
After pulling back \eqref{eq:Theta} and dividing by the maximal common
power $\rho^{j+1}$, the linear part at $(0,z_0)$ is
\[
 \bigl(
  -2\lambda(z-z_0)+c_0\rho
 \bigr)\,\dd\rho
\]
for some $c_0\in\CC$, while the coefficient of $\dd z$ has order at
least two. The associated linear vector field has eigenvalues
\[
 0,\qquad 2\lambda.
\]
Hence this point is a saddle-node, contradicting \textup{(H4)}.

Therefore $\tau\neq0$, and the complementary contribution is one.
Altogether
\[
 I_0(G,T)\leq17+1=18<37.
\]
\end{proof}

It remains only the extremal case
\[
 k=2,
 \qquad
 J_2=cy^2,
 \qquad
 c\neq0.
\]

\section{The extremal degree-seven sector}
\label{sec:d7ext}

We now treat the only degree-seven sector not excluded in the preceding
section. Recall that
\[
 J_2=ax^2+bxy+cy^2,
 \qquad
 k=\ord_{u=0}J_2(1,u),
\]
and that the nonextremal sectors $k=0,1$ have already been ruled out.
Thus only the maximal possible value $k=2$ remains. Equivalently,
\[
 a=b=0,
 \qquad
 c\neq0,
 \qquad
 J_2=cy^2.
\]
This is what we mean by the \emph{extremal degree-seven sector}.

A weighted dilation
\[
 (x,y)\longmapsto(\lambda^3x,\lambda^2y)
\]
preserves the cusp up to a nonzero scalar. In the normal form
\eqref{eq:UV}, the coefficient of $xy^2$ in $U$ is multiplied by
$\lambda^{13}$. Hence, over $\CC$, we may normalize
\[
 c=1.
\]

Thus
\[
 U
 =
 xy^2
 +u_{20}x^2+u_{11}xy+u_{02}y^2
 +u_{10}x+u_{01}y+u_{00},
\]
and
\[
\begin{split}
V={}&-3y^4
 +v_{30}x^3+v_{21}x^2y+v_{12}xy^2+v_{03}y^3\\
&+v_{20}x^2+v_{11}xy+v_{02}y^2
 +v_{10}x+v_{01}y+v_{00}.
\end{split}
\]

The first weighted obstruction equations give
\begin{equation}
 v_{00}=v_{01}=0,
 \qquad
 2v_{03}+3v_{20}=0,
 \qquad
 v_{21}=4.
 \label{eq:d7early}
\end{equation}

For convenience write
\[
 r=v_{03},
 \qquad
 a=v_{12},
 \qquad
 b=v_{11},
 \qquad
 s=v_{02},
 \qquad
 t=v_{10}.
\]
Then
\begin{equation}
\begin{split}
\widehat U={}&
 u^2
 +u_{20}v+u_{11}uv+u_{02}u^2v\\
&+u_{10}v^2+u_{01}uv^2+u_{00}v^3,
\end{split}
\label{eq:Uhatd7}
\end{equation}
and
\begin{equation}
\begin{split}
\widehat V={}&
 -3u^4
 +v(v_{30}+4u+au^2+ru^3)\\
&+v^2\left(-\frac23r+bu+su^2\right)
 +tv^3.
\end{split}
\label{eq:Vhatd7}
\end{equation}

\begin{lemma}[Sharp extremal contact]
\label{lem:d7bezout}
In the extremal degree-seven sector just described, namely
\[
 k=2,
 \qquad
 J_2=cy^2,
 \qquad
 c\neq0,
\]
one has
\[
 I_0(\widehat V,T)=24.
\]
In particular, the projective curves
\[
 (\widehat V=0)
 \quad\text{and}\quad
 (T=0)
\]
of degrees at most $4$ and $6$ attain the global B\'ezout bound at the
origin.
\end{lemma}

\begin{proof}
The identity \eqref{eq:fundamentalidentity} gives modulo $(T)$
\[
 vG=w\widehat V.
\]
In the extremal case $k=2$, equivalently $J_2=cy^2$, one has
\[
 I_0(v,T)=5,
 \qquad
 I_0(w,T)=18.
\]
Therefore
\[
 5+I_0(G,T)
 =
 18+I_0(\widehat V,T).
\]
Since $I_0(G,T)=37$,
\[
 I_0(\widehat V,T)=24.
\]
On the other hand,
\[
 \deg\widehat V\leq4,
 \qquad
 \deg T\leq6.
\]
Thus $24=4\cdot6$ is the B\'ezout maximum. Consequently no additional
intersection of their projective closures can occur away from the
origin.
\end{proof}

Define
\[
 \mathcal R(u)
 =
 \Res_v(T,\widehat V)
 =
 \sum_{\nu\geq0}K_\nu u^\nu.
\]
By Lemma~\ref{lem:d7bezout},
\[
 \ord_{u=0}\mathcal R=24.
\]
Hence
\[
 K_\nu=0
 \qquad
 (\nu<24).
\]

\begin{proposition}[Exact extremal obstruction in degree seven]
\label{prop:d7ext}
There is no foliation in the extremal degree-seven sector
$k=2$, equivalently $J_2=cy^2$ with $c\neq0$, satisfying
\textup{(H1)--(H4)}.
\end{proposition}

\begin{proof}
We use only exact polynomial identities over $\mathbb Q$.

The first relevant resultant coefficient is
\[
 K_5
 =
 v_{30}^2
 \Res_v
 \left(
 v^4+u_{00}v^2+u_{10}v+u_{20},
 \,
 tv^2-\frac23rv+v_{30}
 \right).
\]
If $v_{30}\neq0$, vanishing of the residual resultant would give an
intersection of $(T=0)$ and $(\widehat V=0)$ on $u=0$ with $v\neq0$,
contradicting the equality case of B\'ezout in
Lemma~\ref{lem:d7bezout}. Hence
\[
 v_{30}=0.
\]

We split according to $r=v_{03}$.

\medskip
\noindent
\textbf{Case 1: $r=0$.}

Suppose first that $t=0$. Then $\widehat V=uW$ with $\deg W\leq3$.
Since
\[
 I_0(\widehat V,T)=24
\]
and
\[
 I_0(W,T)\leq18,
\]
one must have
\[
 I_0(u,T)=6.
\]
This forces
\[
 u_{20}=u_{10}=u_{00}=0.
\]
But the next resultant coefficient is
\[
 K_{11}=7168,
\]
a contradiction.

Assume therefore $t\neq0$. The coefficients $K_7,K_8,K_9$ are,
successively,
\[
 K_7=28u_{20}^2t^4,
\]
\[
 K_8=112u_{10}^2t^3,
\]
\[
 K_9=28t^2(4u_{00}-u_{11}t)^2.
\]
Thus
\[
 u_{20}=u_{10}=0,
 \qquad
 u_{00}=\frac14u_{11}t.
\]

The cusp equations of target weights $18,20,24$ now give
\[
 s(2u_{01}+t)=0
\]
together with two further linear relations.

If $s=0$, they give
\[
 u_{11}=0,
 \qquad
 u_{02}=-\frac b2.
\]
Then
\[
 K_{10}=7t^3(t-4u_{01})^2,
\]
so
\[
 u_{01}=\frac t4.
\]
The following coefficients become
\[
 K_{11}
 =
 \frac74(3t^2b-64)^2,
\]
and, after
\[
 b=\frac{64}{3t^2},
\]
\[
 K_{12}
 =
 \frac{7}{144t^5}(3at^5+8192)^2.
\]
Thus
\[
 a=-\frac{8192}{3t^5}.
\]
But then
\[
 K_{13}
 =
 \frac{187904819200}{9t^{10}}
 \neq0,
\]
a contradiction.

It remains to take $s\neq0$. The weight-$18$ equation gives
\[
 u_{01}=-\frac t2.
\]
The equations $K_{10}=0$, the weight-$24$ obstruction, and $K_{11}=0$
give
\[
 u_{11}=-\frac9{64}t^3,
 \qquad
 b=\frac{64}{3t^2},
\]
and
\[
 a
 =
 -\frac{4096s+27t^6}{9st^5}.
\]
Put
\[
 x=\frac{s}{t^6},
 \qquad
 q=t^{13}.
\]
The next resultant condition is
\begin{equation}
 81qx^2+327680x-432=0.
 \label{eq:Pd7}
\end{equation}
The following condition is
\begin{equation}
\begin{split}
0={}&
88086528q^2x^5
-177147q^2x^4
-280850595840qx^4\\
&-4825350144qx^3
+1889568qx^2
-439804651110400x^3\\
&+13480828600320x^2
+95169282048x
-5038848.
\end{split}
\label{eq:Qd7}
\end{equation}
Their exact resultant with respect to $x$ is
\[
 -1001889278026273498595328\,
 q^2(9801q-8388608000)^2.
\]
Since $q\neq0$,
\[
 q=\frac{8388608000}{9801}.
\]
However, the next resultant coefficient, reduced modulo
\eqref{eq:Pd7}--\eqref{eq:Qd7}, is
\[
 \frac{4546773}{5000000}
 (12849111q-43478155264000),
\]
which at the preceding value of $q$ equals
\[
 -\frac{18460290639200256}{625}
 \neq0.
\]
This eliminates the case $r=0$.

\medskip
\noindent
\textbf{Case 2: $r\neq0$.}

After $v_{30}=0$, the next coefficients factor as
\[
 K_7
 =
 \frac{28}{81}u_{20}\,R_7,
\]
and, after $u_{20}=0$,
\[
 K_8
 =
 \frac{28}{243}
 (r^2+6ru_{11}+36u_{10})R_8,
\]
where $R_7,R_8$ are the residual resultants of the corresponding
strict transforms on $u=0$. They cannot vanish because
Lemma~\ref{lem:d7bezout} concentrates the complete B\'ezout intersection at
the origin. Hence
\[
 u_{20}=0,
\]
and
\[
 u_{10}
 =
 -\frac{r(r+6u_{11})}{36}.
 \label{eq:u10first}
\]

Consider the first ordinary blow-up in the chart
\[
 v=uz.
\]
At the relevant nonzero exceptional point
\[
 z_*=\frac6r,
\]
the two exceptional functions are
\[
 \widetilde V(0,z)
 =
 z\left(4-\frac23rz\right),
\]
and
\[
 \widetilde T(0,z)
 =
 z(1+u_{11}z+u_{10}z^2).
\]
Equation \eqref{eq:u10first} is exactly the condition
\[
 \widetilde T(0,z_*)=0.
\]
If
\[
 \widetilde V'(0,z_*)
 \neq
 \widetilde T'(0,z_*),
\]
the transformed foliation has one zero and one nonzero eigenvalue at
this point, hence a saddle-node. Hypothesis \textup{(H4)} therefore
forces equality of these derivatives. Using
\eqref{eq:u10first}, this gives
\[
 -2+\frac{6u_{11}}r=0.
\]
Consequently
\begin{equation}
 u_{11}=\frac r3,
 \qquad
 u_{10}=-\frac{r^2}{12}.
 \label{eq:saddlenode-first}
\end{equation}

We now split according to $s=v_{02}$.

\medskip
\noindent
\emph{Subcase 2a: $s=0$.}

The target-weight-$18$ equation gives
\[
 t=0.
\]
The next cusp equations give
\[
 u_{01}=\frac{rb}{8},
 \qquad
 u_{02}=-\frac b2,
\]
and the contact equations determine
\[
 u_{00}
 =
 \frac{r^3a}{216}+\frac{r^2b}{48}.
\]
The next relevant contact equation is
\[
 8r^3-9r^2a^2-90rab-216b^2=0.
 \label{eq:secondcontact}
\]

At the next blow-up the exceptional quadratic and transverse linear
polynomials may be chosen as
\[
 H(y)
 =
 \frac{18}{r^2}
 \left(
 r^2y^2-3ary-15by+2r
 \right),
\]
and
\[
 L(y)
 =
 12y-\frac{18(ar+6b)}{r^2}.
\]
Equation \eqref{eq:secondcontact} is, up to a nonzero scalar,
\[
 \Res_y(H,L)=0.
\]
Thus $H$ and $L$ have a common exceptional point.

If the corresponding root of $H$ were simple, the transformed
foliation would have a saddle-node. Hence \textup{(H4)} forces this
root to be double.  Here, for a univariate polynomial
\[
 P(t)=a_n\prod_{j=1}^n(t-\alpha_j),
\]
we use the standard discriminant
\[
 \operatorname{Disc}(P)
 =
 a_n^{2n-2}\prod_{i<j}(\alpha_i-\alpha_j)^2,
\]
so that $\operatorname{Disc}(P)=0$ precisely when $P$ has a multiple
root.  Consequently,
\[
 \operatorname{Disc}(H)=0.
\]
Explicitly,
\[
 \Res(H,L)
 =
 648r^2
 \bigl(
 8r^3-9a^2r^2-90abr-216b^2
 \bigr),
\]
whereas
\[
 \operatorname{Disc}(H)
 =
 324
 \bigl(
 8r^3-9a^2r^2-90abr-225b^2
 \bigr).
\]
The two equations imply
\[
 b=0.
\]
Then \eqref{eq:secondcontact} gives
\[
 8r=9a^2.
\]
The next two contact coefficients become
\[
 8019a^{13}+134217728=0,
\]
and
\[
 494343a^{13}+13958643712=0.
\]
These two linear equations in $a^{13}$ are incompatible. Hence
Subcase~2a is impossible.

\medskip
\noindent
\emph{Subcase 2b: $s\neq0$.}

The target weights $18,20,24$, together with
\eqref{eq:saddlenode-first}, give
\[
 u_{01}
 =
 \frac{t(r^2-4s)}{8s},
 \qquad
 u_{02}=-\frac b2,
\]
and
\begin{equation}
 (r^2+12s)(sb-rt)=0.
 \label{eq:splitfinal}
\end{equation}

First suppose
\[
 s=-\frac{r^2}{12}.
\]
The next contact equation and the residual B\'ezout condition give
\[
 ar^2+6rb+36t=0.
\]
The following coefficient gives $t\neq0$ and
\[
 b
 =
 -\frac{32r+15t^3}{rt^2}.
\]
Put
\[
 z=\frac{r}{t^3}.
\]
The next contact coefficient becomes
\begin{equation}
 2048z^2+864z-27=0.
 \label{eq:zquad}
\end{equation}
The subsequent two coefficients give, after eliminating the remaining
parameter,
\begin{equation}
\begin{split}
0={}&
228841226240z^5
+227741270016z^4
+71332724736z^3\\
&+6137579520z^2
-256830912z
+550395.
\end{split}
\label{eq:zpoly}
\end{equation}
The two polynomials in \eqref{eq:zquad} and \eqref{eq:zpoly} are
coprime in $\mathbb Q[z]$. Hence this branch is empty.

It remains to consider
\[
 sb=rt,
 \qquad
 r^2+12s\neq0.
\]
At the second blow-up, the next contact equation is the resultant of
an exceptional quadratic $H$ and a transverse linear polynomial $L$.
At their common root, \textup{(H4)} again forces the root of $H$ to be
double. A direct calculation gives
\[
 L-\frac13H'
 =
 -\frac{18t(r^2+12s)}{r^3s}.
\]
Since $r,s$ and $r^2+12s$ are nonzero, this forces
\[
 t=0.
\]
From $sb=rt$ we obtain
\[
 b=0.
\]
The remaining contact equation gives
\[
 s=\frac{9ra^2-8r^2}{96}.
\]
The next two coefficients are
\[
 11a^3r^5+331776=0,
\]
and
\[
 5095a^3r^5+896ar^6+310542336=0.
\]
Their exact lexicographic Gr\"obner basis contains
\[
 128r-743a^2.
\]
The following contact coefficient reduces modulo these two equations
to
\[
 -\frac{6641574912}{11}a^2.
\]
Thus $a=0$, and then $r=0$, contradicting the standing assumption
$r\neq0$.

All branches have been eliminated. Therefore the extremal degree-seven
sector is impossible.
\end{proof}

\begin{remark}[Exact arithmetic and reproducibility]
\label{rem:certificate}
All resultant expansions, polynomial reductions, and Gr\"obner-basis
computations used in Proposition~\ref{prop:d7ext} are carried out over
$\mathbb Q$.  For reproducibility, the exact computations are provided
in the supplementary file
\[
 \texttt{degree7\_exact\_verification.py}.
\]
Running the script reproduces the polynomial identities and reductions
used in the proof and terminates with the message
\[
 \texttt{Degree-seven exact verification passed.}
\]
The supplementary README records the software requirements and the
command used for verification.  No floating-point computation enters
the certificate.
\end{remark}

\section{Proof of the main theorem}
\label{sec:proof-main}

\begin{proof}[Proof of Theorem~\ref{thm:main}]
By the projective classification of irreducible cuspidal cubics, and
because the cusp and the smooth flex are intrinsic, a projective
automorphism sends the marked triple $(C,p,q)$ to
$(C_0,p_0,q_0)$.  We therefore make the normalization fixed at the end
of the introduction.  In these coordinates, Proposition~\ref{prop:UV}
shows that every foliation satisfying the hypotheses has the form
\[
 \omega=(1+fU)\,\dd f+fV\eta.
\]

If
\[
 K=L=0,
\]
then
\[
 \omega=\dd f
\]
and the foliation is the standard degree-two cusp foliation.

Assume $K,L$ do not both vanish. By Lemma~\ref{lem:J},
\[
 d=m+2,
 \qquad
 m\geq3.
\]
Thus the only possible nonstandard degrees with $d\leq7$ are
\[
 d=5,6,7.
\]

Exact degree five is excluded by Proposition~\ref{prop:d5}, and exact degree six
by Proposition~\ref{prop:d6}.

For degree seven, Lemmas~\ref{lem:d7k0} and~\ref{lem:d7k1} eliminate the two
nonextremal possibilities
\[
 k=0,1.
\]
The only remaining possibility is the extremal sector
\[
 J_2=cy^2,
 \qquad c\neq0,
\]
which is excluded by Proposition~\ref{prop:d7ext}.

Therefore no nonstandard foliation of degree at most seven satisfies
\textup{(H1)--(H4)}. Hence, in the normalized coordinates,
\[
 d=2,
 \qquad
 \omega=\dd(x^2+y^3)
\]
up to multiplication by a nonzero constant.

Returning by the inverse projective automorphism proves that $\F$ is
projectively equivalent to $\F_{\mathrm{cusp}}$.  If $F_C=0$ is a
homogeneous cubic equation for the original curve and $\ell_q=0$ is
its tangent line at the smooth flex, the pull-back of
\[
 \frac{X^2Z+Y^3}{Z^3}
\]
is a nonzero constant multiple of
\[
 \frac{F_C}{\ell_q^3}.
\]
This is therefore a rational first integral of $\F$.
\end{proof}

\begin{remark}[A finite-jet weakening of \textup{(H3)}]
\label{rem:finite-cusp-obstructions}
Inspection of the preceding proof shows that the full analytic existence
of a first integral at $p$ is not required.  The first local input can be
replaced simply by
\[
 \mu_p(\F)=2.
\]
Indeed, write $\omega=A\,\dd f+B\eta$ as in
Lemma~\ref{lem:decomp}, and set $a=A(0)$, $b=B(0)$.  The linear parts of
the two coefficients of $\omega$ are
\[
 2ax-2by,
 \qquad
 3bx.
\]
If $b\neq0$, these two linear forms are independent, since
\[
 \det
 \begin{pmatrix}
 2a&-2b\\
 3b&0
 \end{pmatrix}
 =6b^2\neq0,
\]
and hence $\mu_p(\F)=1$.  If $a=b=0$, both coefficients of $\omega$
belong to the square of the maximal ideal, so their local intersection
multiplicity is at least $4$.  Thus $\mu_p(\F)=2$ forces
$A(0)\neq0$ and $B(0)=0$, and on the normalization
$\gamma(t)=(t^3,-t^2)$ one has
\[
 tA(\gamma(t))+B(\gamma(t))=A(0)t+O(t^2).
\]
This is exactly the simple-zero input used in
Proposition~\ref{prop:KL}.

After the polynomial Bott normal form
\[
 \omega=(1+fU)\,\dd f+fV\eta
\]
is obtained, the remaining use of the first integral is only the
finite-order solvability of the weighted cohomological equation
\[
 (1+fU)X_0(\varphi)+fV E_0(\varphi)=-6fV.
\]
The recurrence of Lemma~\ref{lem:recurrence} shows that, for degrees at
most seven, the only target weights at which a nontrivial solvability
condition is used are
\[
 6,\ 8,\ 12,\ 14,\ 18,\ 20,\ 24.
\]
Thus the argument proves the same conclusion if \textup{(H3)} is
replaced, in the normalized cusp coordinates, by the numerical
condition $\mu_p(\F)=2$ together with solvability of the recurrence of
Lemma~\ref{lem:recurrence} through target weight $24$; equivalently, the
corresponding cokernel obstructions must vanish at the target weights
listed above.  We keep the holomorphic first-integral formulation in
Theorem~\ref{thm:main} because it is intrinsic and geometrically more
transparent.
\end{remark}

\section{Degree-two deformation rigidity on $\PP^2$}
\label{sec:planar-deformation}

We next study the deformation-theoretic counterpart of the preceding
global rigidity theorem. In degree two the conclusion is stronger in
a different direction: once the cuspidal cubic is preserved and the
singular set remains concentrated in two points, no local
first-integrability or saddle-node hypothesis is needed.

Throughout this section we keep the notation
\[
 F(X,Y,Z)=X^2Z+Y^3,
 \qquad
 C=(F=0)\subset\PP^2,
\]
\[
 p=[0:0:1],
 \qquad
 q=[1:0:0],
\]
and, in the affine chart $Z=1$,
\[
 f=x^2+y^3,
 \qquad
 \eta=3x\,\dd y-2y\,\dd x.
\]
The standard cuspidal foliation is
\[
 \F_{\mathrm{cusp}}:\quad \dd f=0,
\]
and a homogeneous defining form is
\[
 \Omega_{\mathrm{cusp}}=Z\,\dd F-3F\,\dd Z.
\]

We first determine all degree-two foliations near
$\F_{\mathrm{cusp}}$ which leave the fixed cubic $C$ invariant.

\begin{lemma}[Normal form for degree-two cusp-preserving foliations]
\label{lem:degree-two-cusp-normal-form-final}
Let \(\F\) be a degree-two foliation on \(\PP^2\),
sufficiently close to \(\F_{\mathrm{cusp}}\), and assume that
\(C\) is invariant by \(\F\). Then, after choosing a nonzero projective defining form and
restricting it to the affine chart \(Z=1\), its affine representative
can be written uniquely in the form
\begin{equation}
\omega
=
\lambda\,\dd f+(a+by+cx)\eta+e\beta,
\label{eq:final-degree-two-normal-form}
\end{equation}
where
\[
\beta=(x^2+3y^3)\,dx-3xy^2\,dy
\]
and
\[
\lambda,a,b,c,e\in\CC.
\]
Moreover, \(\lambda\neq0\) for \(\F\) sufficiently close to
\(\F_{\mathrm{cusp}}\), and \(q\) is singular for
\(\F\) if and only if \(e=0\).
\end{lemma}

\begin{proof}
Write
\[
\omega=P(x,y)\,dx+Q(x,y)\,dy.
\]
Since \(\F\) has degree two, \(P\) and \(Q\) have degree at
most three and, if \(P_3,Q_3\) denote their homogeneous parts of
degree three, the projective condition is
\begin{equation}
xP_3+yQ_3=0.
\label{eq:projective-condition-affine}
\end{equation}
The invariance of \(C=(f=0)\) is equivalent to
\[
f\mid \omega\wedge \dd f.
\]
Since
\[
\dd f=2x\,dx+3y^2\,dy,
\]
this is equivalent to the existence of a polynomial \(H\) of degree
at most two such that
\begin{equation}
3y^2P-2xQ=(x^2+y^3)H.
\label{eq:cusp-invariance-coefficients}
\end{equation}

We now write general polynomials \(P,Q\) of degree at most three and
compare coefficients in
\eqref{eq:projective-condition-affine} and
\eqref{eq:cusp-invariance-coefficients}. The resulting linear system
has dimension five, and its solutions are precisely
\[
\begin{aligned}
P={}&
2\lambda x-2ay-2by^2-2cxy
+e(x^2+3y^3),\\
Q={}&
3\lambda y^2+3ax+3bxy+3cx^2
-3exy^2.
\end{aligned}
\]
This is exactly \eqref{eq:final-degree-two-normal-form}.

For completeness, the corresponding five homogeneous projective
forms may be chosen as
\[
\Omega_{\mathrm{cusp}},
\quad
\Omega_{\eta},
\quad
\Omega_{y\eta},
\quad
\Omega_{x\eta},
\quad
\Omega_{\beta},
\]
where
\[
\begin{aligned}
\Omega_{\eta}
&=-2YZ^2\,dX+3XZ^2\,dY-XYZ\,dZ,\\
\Omega_{y\eta}
&=-2Y^2Z\,dX+3XYZ\,dY-XY^2\,dZ,\\
\Omega_{x\eta}
&=-2XYZ\,dX+3X^2Z\,dY-X^2Y\,dZ,\\
\Omega_{\beta}
&=(X^2Z+3Y^3)\,dX-3XY^2\,dY-X^3\,dZ.
\end{aligned}
\]
The first four forms vanish at \(q=[1:0:0]\), whereas
\(\Omega_{\beta}(q)\neq0\). Hence \(q\) is singular if and only if
\(e=0\). Finally, \(\lambda=1\) at the standard foliation, so
\(\lambda\neq0\) in a sufficiently small neighborhood of it.
\end{proof}

We first record the fixed-singular-set form of rigidity.

\begin{proposition}[Rigidity with fixed singular set]
\label{prop:fixed-singular-set-final}
Let \(\F\) be a degree-two foliation sufficiently close to
\(\F_{\mathrm{cusp}}\). Assume that \((\F,C)\) is a cuspidal pair; that
is, \(C\) is invariant by \(\F\) and
\[
\Sing(\F)=\{p,q\}.
\]
Then
\[
\F=\F_{\mathrm{cusp}}.
\]
\end{proposition}

\begin{proof}
By Lemma~\ref{lem:degree-two-cusp-normal-form-final}, the condition
that \(q\) be singular gives \(e=0\). After multiplication by a
nonzero constant, we may assume \(\lambda=1\). Hence
\[
\omega=\dd f+h\eta,
\qquad
h=a+by+cx.
\]
Its affine coefficients are
\[
2(x-yh),
\qquad
3(y^2+xh),
\]
so the affine singular points satisfy
\begin{equation}
x-yh=0,
\qquad
y^2+xh=0.
\label{eq:final-affine-singular-equations}
\end{equation}
If \(y=0\), the first equation gives \(x=0\). If \(y\neq0\), then
\(x=yh\), and the second equation becomes
\[
y(y+h^2)=0.
\]
Thus every additional affine singularity satisfies
\[
y=-h^2,
\qquad
x=-h^3.
\]
Putting \(u=h\), we obtain
\[
(x,y)=(-u^3,-u^2)
\]
and the identity \(u=a+by+cx\) becomes
\begin{equation}
cu^3+bu^2+u-a=0.
\label{eq:final-extra-singularity-polynomial}
\end{equation}

If \(a\neq0\), the polynomial in
\eqref{eq:final-extra-singularity-polynomial} has a nonzero complex
root. If \(a=0\) and \((b,c)\neq(0,0)\), then
\[
cu^3+bu^2+u=u(cu^2+bu+1),
\]
and the second factor has a nonzero complex root. In either case an
additional affine singular point is produced. Therefore
\[
a=b=c=0,
\]
and \(\omega=\dd f\).
\end{proof}

We can now allow the second singular point to move.

\stepcounter{theorem}
\begin{maintheorem}[Two-singularity deformation rigidity]
\label{thm:two-singularity-deformation-final}
Let
\[
(\F_t)_{t\in(\CC,0)}
\]
be a holomorphic deformation of
\(\F_{\mathrm{cusp}}\) through degree-two foliations on
\(\PP^2\). Assume that, for every sufficiently small \(t\),

\begin{enumerate}
\item the fixed cuspidal cubic \(C\) is invariant by \(\F_t\);
\item \(\Sing(\F_t)\) consists of exactly two
points.
\end{enumerate}
Then
\[
\F_t=\F_{\mathrm{cusp}}
\]
for every sufficiently small \(t\).
\end{maintheorem}

\begin{proof}
For each \(t\), Lemma~\ref{lem:degree-two-cusp-normal-form-final}
gives
\[
\omega_t
=
\lambda\,\dd f+(a+by+cx)\eta+e\beta,
\]
where the coefficients depend holomorphically on \(t\) and
\(\lambda\neq0\).

The point \(p=(0,0)\) is singular for every \(t\). Its linear part is
\[
(2\lambda x-2ay)\,dx+3ax\,dy.
\]
If \(a\neq0\), the corresponding coefficient matrix has determinant
\(6a^2\), so \(p\) is nondegenerate and
\[
\mu(\F_t,p)=1.
\]
For the central foliation,
\[
\mu(\F_{\mathrm{cusp}},p)
=
\dim_{\CC}
\frac{\CC\{x,y\}}{(x,y^2)}
=2.
\]
Hence conservation of the local intersection number implies that,
when \(a\neq0\), another singular point splits off in a sufficiently
small neighborhood of \(p\).

We also record the Milnor number at \(q\). In the chart \(X=1\), with
\[
u=Y/X,\qquad v=Z/X,
\]
the standard foliation is represented by
\[
3vu^2\,du-(2v+3u^3)\,dv.
\]
Therefore
\[
\mu(\F_{\mathrm{cusp}},q)
=
\dim_{\CC}
\frac{\CC\{u,v\}}
     {(vu^2,\,2v+3u^3)}
=
\dim_{\CC}\frac{\CC\{u\}}{(u^5)}
=5.
\]
Again by conservation of number, every sufficiently small deformation
has at least one singular point in a small neighborhood of \(q\).
Consequently \(a\neq0\) would give at least three distinct singular
points globally, contrary to the hypothesis. Hence
\[
a=0.
\]

Parametrize the smooth affine part of the cusp by
\[
x=u^3,\qquad y=-u^2.
\]
Substitution in the two coefficients of \(\omega_t\) shows that,
away from \(u=0\), a point of \(C\) is singular precisely when
\begin{equation}
H(u)=eu^3-cu^2+bu-\lambda=0.
\label{eq:final-H-polynomial}
\end{equation}

Suppose first that \(e=0\). Then \(q\) remains singular. If
\((b,c)\neq(0,0)\), the nonconstant polynomial
\[
-cu^2+bu-\lambda
\]
has a finite complex root. Since \(\lambda\neq0\), this root is
nonzero and yields a singular point on the smooth affine part of
\(C\), in addition to \(p\) and \(q\). Hence \(b=c=0\), and
Proposition~\ref{prop:fixed-singular-set-final} gives
\[
\F_t=\F_{\mathrm{cusp}}.
\]

It remains to exclude \(e\neq0\). In this case \(q\) is regular.
Since \(H\) is a cubic polynomial, all its roots are finite; moreover
\(H(0)=-\lambda\neq0\). Each root therefore gives a singular point on
the smooth part of \(C\). The parametrization
\[
u\longmapsto(u^3,-u^2)
\]
is injective for \(u\neq0\). Since there is only one singular point
besides \(p\), all three roots of \(H\), counted with multiplicity,
must coincide. Hence
\[
H(u)=e(u-r)^3
\]
for some \(r\neq0\). Comparing coefficients gives
\[
c=3er,\qquad
b=3er^2,\qquad
\lambda=er^3.
\]
Normalize \(\lambda=1\), and put \(s=1/r\). Then
\[
e=s^3,\qquad c=3s^2,\qquad b=3s.
\]

Introduce the rescaled affine coordinates
\[
X=s^3x,\qquad Y=s^2y.
\]
The full singularity equations become
\begin{equation}
\begin{cases}
X^2-6XY+2X+3Y^3-6Y^2=0,\\
3X^2-XY^2+3XY+Y^2=0.
\end{cases}
\label{eq:final-rescaled-singular-system}
\end{equation}
Besides the origin and the point
\[
(X,Y)=(1,-1)
\]
lying on the cusp, the two points
\[
(X,Y)=(-2,\,1+i\sqrt3),
\qquad
(X,Y)=(-2,\,1-i\sqrt3)
\]
also satisfy \eqref{eq:final-rescaled-singular-system}. They are
distinct and do not lie on the cusp, since
\[
X^2+Y^3=-4
\]
at both points. Thus the foliation has at least four singular points,
contradicting the hypothesis. Therefore \(e=0\), and the preceding
case completes the proof.
\end{proof}

The preceding theorem has a useful Milnor-number formulation.

\begin{corollary}[Cusp persistence and no splitting at infinity]
\label{cor:cusp-no-splitting-final}
Let
\[
(\F_t)_{t\in(\CC,0)}
\]
be a holomorphic deformation of
\(\F_{\mathrm{cusp}}\) through degree-two foliations on
\(\PP^2\), all leaving \(C\) invariant. Assume that, for every
sufficiently small \(t\),

\begin{enumerate}
\item the cusp point $p$ is a singular point of $\F_t$ and
\[
\mu(\F_t,p)=2;
\]
\item there exists a singular point \(q_t\to q\) such that
\[
\mu(\F_t,q_t)=5.
\]
\end{enumerate}

Then
\[
\F_t=\F_{\mathrm{cusp}}
\]
for every sufficiently small \(t\).
\end{corollary}

\begin{proof}
For every degree-two foliation on \(\PP^2\) with isolated
singularities,
\[
\sum_{r\in\Sing(\F_t)}
\mu(\F_t,r)
=
2^2+2+1
=
7.
\]
The two singularities in the hypotheses already contribute
\[
2+5=7.
\]
Since the Milnor number of every isolated singularity is positive,
there can be no further singular point. Thus
\[
\Sing(\F_t)=\{p,q_t\},
\]
and Theorem~\ref{thm:two-singularity-deformation-final} applies.
\end{proof}

We now give two examples which show that the hypotheses above are
essential.

\begin{example}[Preserving the cubic does not preserve two singularities]
\label{ex:alpha-deformation-final}
Consider
\[
\omega_t=\dd f+t\eta.
\]
This is the affine restriction of the degree-two projective family
\[
\Omega_t
=
\Omega_{\mathrm{cusp}}+t\Omega_{\eta},
\]
and \(C\) is invariant for every \(t\). Since \(e=0\), the point \(q\)
remains singular. The affine singularity equations are
\[
x-ty=0,
\qquad
y^2+tx=0,
\]
and for \(t\neq0\) they give
\[
p=(0,0),
\qquad
p_t=(-t^3,-t^2).
\]
On the line at infinity the only singular point is \(q\). Thus
\[
\Sing(\F_t)
=
\{p,p_t,q\}.
\]
Moreover \(p\) and \(p_t\) are nondegenerate, while the singularity at
\(q\) retains Milnor number \(5\). Hence the Milnor numbers distribute
as
\[
1+1+5=7.
\]
This shows that invariance of the cuspidal cubic alone does not imply
rigidity.
\end{example}

\begin{example}[Persistence of the affine cusp alone is not sufficient]
\label{ex:yalpha-deformation-final}
Consider
\[
\omega_t=\dd f+t\,y\eta
=
2(x-ty^2)\,dx+3(y^2+txy)\,dy.
\]
Again this is the affine restriction of a degree-two projective family
preserving \(C\). At the origin,
\[
\mu(\F_t,p)
=
\dim_{\CC}
\frac{\CC\{x,y\}}
     {(x-ty^2,\,y^2+txy)}
=2.
\]

In fact the reduction type of the cusp singularity is preserved.
Indeed, in the first blow-up chart
\[
x=uy
\]
the saturated pull-back is
\[
-2y(ty-u)\,du
+
(tuy+2u^2+3y)\,dy.
\]
At the remaining singular point of the total transform, put
\[
y=uv.
\]
After saturation we obtain
\[
-v(tuv-4u-3v)\,du
+
u(tuv+2u+3v)\,dv.
\]
The strict transform of the cusp, together with the two exceptional
components, still meets at the origin. A third blow-up,
\[
v=uz,
\]
gives the saturated form
\begin{equation}
6z(z+1)\,du
+
u(tuz+3z+2)\,dz.
\label{eq:final-yalpha-last-blowup}
\end{equation}
On the new exceptional divisor \(u=0\), the only singular points are
\[
z=0,\qquad z=-1,
\]
and both are nondegenerate for all sufficiently small \(t\).
Consequently the same sequence of three blow-ups reduces the cusp
singularity for all sufficiently small $t$: the strict transform of
$C$ and the exceptional divisor have the same normal-crossing
configuration, and the singularities appearing on the last exceptional
component are nondegenerate.  In this concrete sense the reduction type
at the cusp is preserved.

Nevertheless, for \(t\neq0\), there is an additional affine
singularity
\[
p_t=
\left(\frac1{t^3},-\frac1{t^2}\right)\in C.
\]
There is also an additional singularity on the line at infinity,
\[
r_t=[1:-t/3:0].
\]
The original point \(q=[1:0:0]\) remains singular. In the chart
\(X=1\), with \(u=Y/X\) and \(v=Z/X\), the family is represented near
\(q\) by
\[
3uv(u+t)\,du
-
(2v+tu^2+3u^3)\,dv.
\]
For \(t\neq0\), \(u+t\) is a unit near \(q\), and therefore
\[
\mu(\F_t,q)
=
\dim_{\CC}
\frac{\CC\{u,v\}}
     {(uv,\,2v+tu^2+3u^3)}
=
\dim_{\CC}\frac{\CC\{u\}}{(u^3)}
=3.
\]
A direct Jacobian computation shows that \(p_t\) and \(r_t\) are
nondegenerate. Thus
\[
2+3+1+1=7.
\]
Both \(p_t\) and \(r_t\) converge projectively to \(q\) as
\(t\to0\). Hence the affine cusp can retain the same reduction type while the
Milnor-number-five singularity at infinity splits.
\end{example}

\begin{remark}[The two escape mechanisms]
\label{rem:two-escape-mechanisms-final}
Examples~\ref{ex:alpha-deformation-final} and
\ref{ex:yalpha-deformation-final} display the two basic ways in which
a degree-two cusp-preserving deformation can escape rigidity. In the
first family the affine cusp singularity itself splits. In the second
family the affine cusp retains the same reduction type, but the singularity at
infinity splits. Thus a nontrivial deformation preserving \(C\) must
lose at least one of the two concentration properties of the standard
foliation.
\end{remark}

\section{Cuspidal linear pull-back loci in higher dimension}
\label{sec:pullback-locus}

We now apply the planar deformation theorem to linear pull-backs of the
standard cuspidal foliation. The resulting family is not a new
irreducible component of the ambient moduli space; rather, it is a
distinguished cuspidal sublocus inside the classical linear pull-back
component.

Let \(n\ge3\) and let
\[
\pi:\PP^n\dashrightarrow\PP^2,
\qquad
\pi=[L_0:L_1:L_2],
\]
be a dominant linear rational map, where
\(L_0,L_1,L_2\) are linearly independent linear forms. Its center is
\[
K_\pi=(L_0=L_1=L_2=0)\simeq\PP^{n-3}.
\]
We put
\[
\mathcal G_\pi=\pi^*\F_{\mathrm{cusp}}.
\]
The invariant cubic pulls back to the cuspidal cone
\[
\widehat C_\pi
=
(L_0^2L_2+L_1^3=0).
\]
Define also
\[
P_{p,\pi}=(L_0=L_1=0),
\qquad
P_{q,\pi}=(L_1=L_2=0).
\]
Then
\[
P_{p,\pi}\simeq P_{q,\pi}\simeq\PP^{n-2},
\qquad
P_{p,\pi}\cap P_{q,\pi}=K_\pi.
\]

\begin{lemma}[Singular set of the cuspidal pull-back]
\label{lem:singular-set-pullback-final}
For every dominant linear rational map \(\pi\),
\[
\Sing(\mathcal G_\pi)
=
P_{p,\pi}\cup P_{q,\pi}.
\]
At a general point of \(P_{p,\pi}\), respectively \(P_{q,\pi}\), a
two-dimensional transverse section has the same singularity as the
planar model at \(p\), respectively at \(q\).
\end{lemma}

\begin{proof}
Outside the center \(K_\pi\), the map \(\pi\) is a submersion.
Therefore
\[
\pi^*\Omega_{\mathrm{cusp}}(x)=0
\]
if and only if
\[
\Omega_{\mathrm{cusp}}(\pi(x))=0.
\]
Since
\[
\Sing(\F_{\mathrm{cusp}})=\{p,q\},
\]
the singular set outside \(K_\pi\) is the union of the two fibers over
\(p\) and \(q\). Their closures are precisely
\(P_{p,\pi}\) and \(P_{q,\pi}\). The coefficients of
\(\pi^*\Omega_{\mathrm{cusp}}\) also vanish on the center, and
\[
K_\pi=P_{p,\pi}\cap P_{q,\pi}.
\]
This proves the formula for the singular set.

At a general point of either component, choose local coordinates
\[
(z_1,z_2,w_1,\ldots,w_{n-2})
\]
such that \(\pi\) is locally the projection onto
\((z_1,z_2)\). Then \(\mathcal G_\pi\) is defined by the pull-back of
the corresponding planar germ, so every transverse two-plane
\((w_1=\cdots=w_{n-2}=0)\) carries exactly that planar singularity.
\end{proof}

\begin{definition}[Cuspidal linear pull-back locus]
\label{def:cuspidal-linear-pullback-locus-final}
For \(n\ge3\), set
\[
\mathcal C_n^{\mathrm{cusp}}
=
\overline{
\left\{
\pi^*\F_{\mathrm{cusp}};
\ \pi:\PP^n\dashrightarrow\PP^2
\text{ dominant and linear}
\right\}}
\subset
\operatorname{Fol}_2(\PP^n),
\]
where the bar denotes Zariski closure. We call
\(\mathcal C_n^{\mathrm{cusp}}\) the \emph{cuspidal linear pull-back locus}.
\end{definition}

\begin{proposition}
\label{prop:cuspidal-locus-irreducible-final}
The locus \(\mathcal C_n^{\mathrm{cusp}}\) is a proper closed irreducible
algebraic subset of \(\operatorname{Fol}_2(\PP^n)\), contained in the
classical linear pull-back irreducible component.
\end{proposition}

\begin{proof}
The set of triples
\[
(L_0,L_1,L_2)
\]
of linearly independent linear forms is a nonempty Zariski open subset
of the irreducible vector space
\[
H^0(\PP^n,\mathcal O_{\PP^n}(1))^{\oplus3}.
\]
The assignment
\[
(L_0,L_1,L_2)
\longmapsto
[\pi^*\Omega_{\mathrm{cusp}}]
\]
is algebraic. Hence the Zariski closure of its image is irreducible.

To identify its position in the space of foliations, consider the
larger irreducible parameter space of pairs
\[
(\pi,\F),
\qquad
\F\in\operatorname{Fol}_2(\PP^2),
\]
with \(\pi\) dominant and linear. The closure of the image of the
pull-back map
\[
(\pi,\F)\longmapsto \pi^*\F
\]
is the classical linear pull-back component. The subfamily obtained
by fixing \(\F=\F_{\mathrm{cusp}}\) is precisely the
family defining \(\mathcal C_n^{\mathrm{cusp}}\). Hence
\(\mathcal C_n^{\mathrm{cusp}}\) is contained in that component.

The inclusion is proper.  Indeed, the family defining
\(\mathcal C_n^{\mathrm{cusp}}\) is parametrized by the nonempty open
set of linearly independent triples in
\[
 \PP\!\left(H^0(\PP^n,\mathcal O_{\PP^n}(1))^{\oplus3}\right),
\]
which has dimension \(3(n+1)-1=3n+2\).  Therefore
\[
 \dim \mathcal C_n^{\mathrm{cusp}}\leq 3n+2.
\]
On the other hand, Ferrer and Vainsencher prove that the linear
pull-back component satisfies
\[
 \dim LPB(d,n)
 =
 \dim \mathbb G(3,n+1)+(d+1)(d+3)-1;
\]
see \cite[Proposition~2.1]{FerrerVainsencher2021}.  For \(d=2\), since
\(\dim\mathbb G(3,n+1)=3(n-2)\), this gives
\[
 \dim LPB(2,n)=3n+8.
\]
Thus \(\mathcal C_n^{\mathrm{cusp}}\) is strictly contained in the
linear pull-back component.
\end{proof}

Thus \(\mathcal C_n^{\mathrm{cusp}}\) is naturally viewed as a
distinguished cuspidal sublocus of the linear pull-back component,
rather than as a new irreducible component of the full space of
degree-two foliations.  Here there is no conflict between
irreducibility and this conclusion: an irreducible component is a
\emph{maximal} irreducible closed subset of the ambient algebraic
space, whereas \(\mathcal C_n^{\mathrm{cusp}}\) is contained in the
larger irreducible linear pull-back component.

\begin{remark}
For \(n\ge3\), Cerveau and Lins Neto proved that the space of
degree-two codimension-one foliations on \(\PP^n\) has six
irreducible components; see
\cite{CerveauLinsNeto1996}. Linear pull-backs form one of the
classical component families; see also
\cite{CerveauLinsNetoEdixhoven2001,FerrerVainsencher2021}.
The proper closed irreducible algebraic subset
\(\mathcal C_n^{\mathrm{cusp}}\) is therefore naturally viewed as a
geometrically distinguished locus inside one of these known
irreducible components.  This motivates the terminology
\emph{cuspidal linear pull-back locus} and places the construction in
the broader setting of studying natural irreducible algebraic loci
inside irreducible components of spaces of foliations.
\end{remark}

The planar theorem immediately gives rigidity in the base direction
for deformations which are already known to remain linear pull-backs.

\stepcounter{theorem}
\begin{maintheorem}[Rigidity inside the linear pull-back locus]
\label{thm:rigidity-inside-pullback-final}
Let \(n\ge3\), and let
\[
(\mathcal G_t)_{t\in(\CC,0)}
\]
be a holomorphic deformation of
\[
\mathcal G_0=\pi_0^*\F_{\mathrm{cusp}}
\]
of the form
\[
\mathcal G_t=\pi_t^*\F_t,
\]
where \(\pi_t:\PP^n\dashrightarrow\PP^2\) is a
holomorphic family of dominant linear rational maps and
\(\F_t\) is a holomorphic family of degree-two foliations on
\(\PP^2\), with
\[
\F_0=\F_{\mathrm{cusp}}.
\]

Assume that there is a holomorphic family of irreducible cuspidal
cubics \(C_t\subset\PP^2\), with \(C_0=C\), such that, for every
sufficiently small \(t\),

\begin{enumerate}
\item \(C_t\) is invariant by \(\F_t\);
\item \(\Sing(\F_t)\) consists of exactly two
points.
\end{enumerate}

Then, after shrinking the parameter disc, there exists a holomorphic
family
\[
A_t\in\operatorname{PGL}(3,\CC),
\qquad
A_0=\operatorname{id},
\]
such that
\[
\F_t=A_t^*\F_{\mathrm{cusp}}.
\]
Consequently
\[
\mathcal G_t
=
(A_t\circ\pi_t)^*\F_{\mathrm{cusp}},
\]
and therefore
\[
\mathcal G_t\in\mathcal C_n^{\mathrm{cusp}}
\]
for every sufficiently small \(t\).
\end{maintheorem}

\begin{proof}
Every irreducible cuspidal cubic in \(\PP^2\) is projectively
equivalent to \(C\); see \cite[Chapter~3]{Dolgachev2012}. The family \(C_t\), being a small holomorphic
family in the \(\operatorname{PGL}(3,\CC)\)-orbit of \(C\),
admits, after shrinking the parameter disc, a holomorphic lift
\[
A_t\in\operatorname{PGL}(3,\CC),
\qquad
A_0=\operatorname{id},
\]
such that
\[
A_t(C_t)=C.
\]
Indeed, one may first lift the family \(C_t\) by local holomorphic
sections of the orbit map and then take inverses. Equivalently, this
uses the local holomorphic triviality of the homogeneous
\(\operatorname{PGL}(3,\CC)\)-orbit of \(C\). The relevant orbit map is
\[
\operatorname{PGL}(3,\CC)
\longrightarrow
\operatorname{PGL}(3,\CC)\cdot C.
\]

Set
\[
\widetilde{\F}_t
=
(A_t^{-1})^*\F_t.
\]
Then \(\widetilde{\F}_t\) is a degree-two deformation of
\(\F_{\mathrm{cusp}}\), it leaves the fixed cubic \(C\)
invariant, and it has exactly two singular points. Therefore
Theorem~\ref{thm:two-singularity-deformation-final} gives
\[
\widetilde{\F}_t
=
\F_{\mathrm{cusp}}.
\]
Thus
\[
\F_t=A_t^*\F_{\mathrm{cusp}},
\]
and hence
\[
\mathcal G_t
=
\pi_t^*\F_t
=
(A_t\circ\pi_t)^*\F_{\mathrm{cusp}}.
\]
\end{proof}

\begin{corollary}[Milnor-number version on the base]
\label{cor:pullback-milnor-version-final}
Keep the pull-back setup and the holomorphic invariant family of
cuspidal cubics \(C_t\) from
Theorem~\ref{thm:rigidity-inside-pullback-final}. Replace the
two-singularity assumption by the following conditions:

\begin{enumerate}
\item the cusp point $p_t\in C_t$ is a singular point of $\F_t$ with
\[
\mu(\F_t,p_t)=2;
\]
\item there exists a singular point \(q_t\neq p_t\) with
\[
\mu(\F_t,q_t)=5.
\]
\end{enumerate}

Then
\[
\mathcal G_t
=
(A_t\circ\pi_t)^*\F_{\mathrm{cusp}}
\]
for a holomorphic family
\(A_t\in\operatorname{PGL}(3,\CC)\).
\end{corollary}

\begin{proof}
The global Milnor-number formula on \(\PP^2\) gives
\[
\sum_r\mu(\F_t,r)=7.
\]
The two prescribed singularities already contribute \(2+5=7\), so
there are no others. Theorem~\ref{thm:rigidity-inside-pullback-final} applies.
\end{proof}

\begin{remark}[Higher-dimensional counterexamples]
\label{rem:higher-dimensional-counterexamples-final}
The two planar counterexamples immediately pull back to every
dimension \(n\ge3\).

If
\[
\mathcal G_t
=
\pi^*(\dd f+t\eta),
\]
then the cuspidal cone \(\widehat C_\pi\) remains invariant, but the
singular set has three codimension-two components, namely the closures
of the fibers over
\[
p,\qquad p_t,\qquad q.
\]
Thus preservation of the cuspidal cone alone does not imply rigidity.

If instead
\[
\mathcal G_t
=
\pi^*(\dd f+t\,y\eta),
\]
then the transverse reduction type along the component over \(p\) remains
unchanged, but the singularity over \(q\) splits. For \(t\neq0\), the
singular set has the four codimension-two components lying over
\[
p,\qquad p_t,\qquad q,\qquad r_t.
\]
Therefore even preservation of the invariant cuspidal cone together
with persistence of the same transverse reduction type along
\(P_{p,\pi}\) is not sufficient.
\end{remark}

\begin{remark}[Why the projection is allowed to move]
Even after the planar foliation has been rigidified, one cannot expect
\[
\mathcal G_t=\mathcal G_0
\]
as foliations on the fixed space \(\PP^n\), since the linear
projection itself may vary. The correct rigidity conclusion is
\[
\mathcal G_t=\pi_t'^*\F_{\mathrm{cusp}}
\]
for a suitable family of linear projections \(\pi_t'\). Thus the
cuspidal locus is rigid in the base-foliation direction, while the
natural parameters of the projection remain.
\end{remark}

\subsection{An intrinsic pull-back recognition problem}

Theorem~\ref{thm:rigidity-inside-pullback-final} assumes from the
outset that the deformation remains in the linear pull-back locus. It
is natural to ask whether this assumption follows from the intrinsic
singular geometry of the cuspidal model.

\begin{problem}[Intrinsic pull-back recognition]
\label{prob:intrinsic-cusp-pullback-final}
Let \(n\ge3\), and let
\[
(\mathcal G_t)_{t\in(\CC,0)}
\]
be a sufficiently small holomorphic deformation in
\(\operatorname{Fol}_2(\PP^n)\) of
\[
\mathcal G_0=\pi_0^*\F_{\mathrm{cusp}}.
\]
Assume that, for every sufficiently small \(t\),

\begin{enumerate}
\item \(\mathcal G_t\) leaves invariant a hypersurface
\(\widehat C_t\) projectively equivalent to a linear cuspidal cone
\[
(L_0^2L_2+L_1^3=0);
\]
\item
\[
\Sing(\mathcal G_t)
=
P_{p,t}\cup P_{q,t},
\]
where \(P_{p,t}\) and \(P_{q,t}\) are codimension-two linear
subspaces meeting along a codimension-three linear subspace
\[
K_t=P_{p,t}\cap P_{q,t};
\]
\item at a general point of each component, a transverse
      two-dimensional germ is analytically equivalent to the corresponding
      planar singularity: the germ at \(p\) along \(P_{p,t}\), and the germ
      at \(q\) along \(P_{q,t}\).
\end{enumerate}

Must there exist a holomorphic family of dominant linear rational maps
\[
\pi_t:\PP^n\dashrightarrow\PP^2
\]
such that
\[
\mathcal G_t=\pi_t^*\F_{\mathrm{cusp}}?
\]
\end{problem}

A positive answer would characterize the cuspidal locus
\(\mathcal C_n^{\mathrm{cusp}}\) intrinsically, using only the invariant cuspidal cone
and the two distinguished transverse singularity types. The two counterexamples above motivate the need to control both singular
components; in particular, the second shows that preservation of the reduction
type along the cusp component alone is insufficient.

A possible approach is to exploit the persistence and unfolding
geometry of these two codimension-two singular components.  Their
intersection
\[
 K_t=P_{p,t}\cap P_{q,t}
\]
already determines a candidate linear projection, up to a projective
change of coordinates on the target.  After blowing up $K_t$, the
recognition problem becomes one of proving that the transformed
foliation contains the relative tangent directions of the resulting
morphism to $\PP^2$.  Methods relating persistent singular schemes,
Kupka schemes and first-order unfoldings may be useful in this step;
see \cite{MassriMolinuevoQuallbrunn2018,Perrella2024}.  We stress,
however, that the two distinguished components are not ordinary Kupka
components for the standard cusp model, since $d\Omega_{\mathrm{cusp}}$
vanishes over both $p$ and $q$.  Thus any such argument must use a
persistent-singularity or unfolding version of the Kupka phenomenon,
not the classical Kupka persistence theorem directly.

\begin{remark}[Relation with general pull-back stability]
\label{rem:general-stability-does-not-directly-apply-final}
Known stability results show that broad classes of pull-back
foliations remain pull-backs under deformation. However, the standard
hypotheses of the available general results do not immediately cover
the present special cusp model. For example, a recent pull-back
stability theorem assumes, among other conditions, that the base
foliation have split tangent sheaf with non-positive splitting and
that the zero set of the exterior derivative of a base defining form
have codimension strictly greater than two; it also requires
\(n\ge m+2\), hence \(n\ge4\) for a two-dimensional base; see
\cite{GargiuloMolinuevoQuallbrunnVelazquez2026}. Here
\[
d\Omega_{\mathrm{cusp}}
=
4\,dZ\wedge dF,
\]
and therefore
\[
\Sing(d\Omega_{\mathrm{cusp}})
=
\{p,q\},
\]
which has codimension two in \(\PP^2\). Thus the intrinsic
recognition problem above is not an immediate consequence of that
general theorem and requires a separate argument adapted to the cusp
geometry.
\end{remark}

Proposition~\ref{prop:cuspidal-locus-irreducible-final} suggests a
broader question concerning the internal algebraic geometry of the
irreducible components of spaces of projective foliations.  Once an
irreducible component is known, one may ask for the geometrically or
dynamically distinguished irreducible algebraic subloci that it
contains.

\begin{problem}[Natural algebraic subloci of foliation components]
Let $\mathcal X$ be an irreducible component of
$\operatorname{Fol}_d(\PP^n)$.  Describe and study the geometrically or
dynamically distinguished irreducible locally closed algebraic loci
$Z\subset\mathcal X$.  In particular, determine their dimensions and
incidence relations, the geometry of their Zariski closures and boundary
strata, and the extent to which their defining geometric properties are
stable under deformation.
\end{problem}

\end{document}